\documentclass[a4paper]{amsart}
\usepackage{amsmath, amsthm, amsfonts, mathtools, tikz-cd, amssymb}
\usepackage{enumitem}
\usepackage{hyperref}

\title[A note on finiteness properties of vertex stabilisers]{A note on finiteness properties of \\ vertex stabilisers}
\date{February~20, 2025}

\makeatletter
\@namedef{subjclassname@2020}{%
\textup{2020} Mathematics Subject Classification}
\makeatother

\author[K.~Li]{Kevin Li}
\address{Fakult\"at f\"ur Mathematik, Universit\"at Regensburg, 93040 Regensburg, Germany}
\email{kevin.li@ur.de}

\author[L.~S\'anchez Salda\~na]{Luis Jorge S\'anchez Salda\~na}
\address{Departamento de Matem\'aticas, Facultad de Ciencias, Universidad Nacional Aut\'onoma de M\'exico, 04510 Ciudad de M\'exico, Mexico}
\email{luisjorge@ciencias.unam.mx}

\theoremstyle{definition}
\newtheorem{defn}{Definition}[section]
\newtheorem{ex}[defn]{Example}
\newtheorem{rem}[defn]{Remark}
\newtheorem{question}[defn]{Question}
\newtheorem*{ack}{Acknowledgements}
\theoremstyle{plain}
\newtheorem{thm}[defn]{Theorem}
\newtheorem{lem}[defn]{Lemma}
\newtheorem{prop}[defn]{Proposition}
\newtheorem{cor}[defn]{Corollary}

\newcommand{\enum}{\rm{(\roman*)}}
\newcommand{\spann}[1]{{\ensuremath \langle{#1}\rangle}}

\newcommand{\IN}{\ensuremath{\mathbb{N}}}
\newcommand{\IR}{\ensuremath{\mathbb{R}}}
\newcommand{\IZ}{\ensuremath{\mathbb{Z}}}
\newcommand{\sfF}{\ensuremath{\mathsf{F}}}
\newcommand{\sfFP}{\ensuremath{\mathsf{FP}}}
\newcommand{\calC}{\ensuremath{\mathcal{C}}}
\newcommand{\calH}{\ensuremath{\mathcal{H}}}
\newcommand{\calR}{\ensuremath{\mathcal{R}}}

\DeclareMathOperator{\Ind}{Ind}
\DeclareMathOperator{\lk}{Lk}
\DeclareMathOperator{\st}{St}

\begin{document}

\begin{abstract}
	We prove a criterion for the geometric and algebraic finiteness properties of vertex stabilisers of $G$-CW-complexes, given the finiteness properties of the group~$G$ and of the stabilisers of positive dimensional cells.
	This generalises a result of Haglund--Wise for groups acting on trees to higher dimensions.
	As an application, for~$n\ge 2$, we deduce the existence of uncountably many quasi-isometry classes of one-ended groups that are of type~$\mathsf{FP}_n$ and not of type~$\mathsf{FP}_{n+1}$.
\end{abstract}

\maketitle

\section{Introduction}
Finiteness properties of discrete groups originate from and strengthen the fundamental notions of being finitely generated and being finitely presented.
The study of higher finiteness properties uses a multitude of methods, e.g., from equivariant topology, homological algebra, and geometric group theory. 
	
Let~$n\in \IN$.
A group~$G$ is said to be \emph{of type~$\sfF_n$} if there exists a $G$-CW-model for~$EG$ (i.e., a contractible free $G$-CW-complex) with cocompact $n$-skeleton.
For example, every group is of type~$\sfF_0$, and a group~$G$ is of type~$\sfF_1$ (resp.~$\sfF_2$) if and only if~$G$ is finitely generated (resp.\ finitely presented).
Finite groups are of type~$\sfF_n$ for all~$n\in \IN$.
The following is a classical combination theorem.

\begin{thm}[\cite{Brown87}]
\label{thm:combination}
	Let~$G$ be a group, let~$X$ be a $G$-CW-complex, and let $n\in \IN$.
	Suppose that the following hold:
	\begin{enumerate}[label=\enum]
		\item $X$ is $(n-1)$-connected;
		\item The quotient~$G\backslash X^{(n)}$ of the $n$-skeleton is compact;
		\item For all~$i\in \{0,\ldots,n-1\}$, all stabilisers of~$i$-cells of~$X$ are of type~$\sfF_{n-i}$.
	\end{enumerate}
	Then~$G$ is of type~$\sfF_n$.
\end{thm}

In this note, we prove a criterion for the finiteness properties of vertex stabilisers. We consider simplicial $G$-CW-complexes (i.e., $G$-CW-complexes whose underlying CW-complex is a simplicial complex) and impose conditions on links of vertices.

\begin{thm}
	\label{thm:main Fn}
	Let~$G$ be a group, let~$X$ be a simplicial $G$-CW-complex, and let $n\in \IN$.
	Suppose that the following hold:
	\begin{enumerate}[label=\enum]
		\item $X$ is $n$-connected;
		\item The quotient~$G\backslash X^{(n+1)}$ of the $(n+1)$-skeleton is compact;
		\item For all~$i\in \{1,\ldots,n\}$, all stabilisers of $i$-simplices of~$X$ are of type~$\sfF_{n+1-i}$;
		\item\label{item:main link} For every 0-simplex~$v$ of~$X$, every connected component of the link of~$v$ in~$X$ is simply-connected; 
		\item $G$ is of type~$\sfF_n$.
	\end{enumerate}
	Then all stabilisers of 0-simplices of~$X$ are of type~$\sfF_n$.
\end{thm}

In the case that~$X$ is 1-dimensional, Theorem~\ref{thm:main Fn} recovers and provides a different proof of a result of Haglund--Wise~\cite{Haglund-Wise21} for groups acting on trees. 

\begin{cor}[\cite{Haglund-Wise21}]
\label{cor:Haglund-Wise}
	Let~$G$ be the fundamental group of a finite graph of groups and let~$n\in \IN$.
	If~$G$ is of type~$\sfF_n$ and all edge groups are of type~$\sfF_n$, then all vertex groups are of type~$\sfF_n$.
\end{cor}

We prove Theorem~\ref{thm:main Fn} by establishing two corresponding criteria for the algebraic finiteness properties~$\sfFP_n$ (Proposition~\ref{prop:FPn}) and for finite presentability (Proposition~\ref{prop:F2}).
The assumption~\ref{item:main link} in Theorem~\ref{thm:main Fn} is needed only for the latter.
We do know whether Theorem~\ref{thm:main Fn} remains true when the assumption~\ref{item:main link} is dropped (Question~\ref{qn:F2}).

As an application of the algebraic criterion (Proposition~\ref{prop:FPn}), for~$n\ge 2$, we deduce the existence of uncountably many quasi-isometry classes of one-ended groups that are of type~$\sfFP_n$ and not of type~$\sfFP_{n+1}$ (Theorem~\ref{thm:uncountably many}).

\begin{ack}
	We thank Clara L\"oh and Matt Zaremsky for helpful discussions.
	This work was started during a stay of the authors at the University of Oxford.
	We thank Sam Hughes and Ric Wade for their hospitality.
	
	The first author was supported by the SFB~1085 \emph{Higher Invariants} (Universit\"at Regensburg, funded by the DFG).
	The second author is grateful for the financial support of DGAPA-UNAM grant PAPIIT IA106923. 
\end{ack}

\section{Algebraic finiteness properties of vertex stabilisers}
For background on finiteness properties we refer to Brown's book~\cite[Chapter~VIII]{Brown82}.
Throughout, let~$R$ be a ring and let~$n\in \IN$.
	
An $R$-module~$M$ is said to be \emph{of type~$\sfFP_n$} if there exists a projective resolution~$P_*$ of~$M$ over~$R$ such that the projective $R$-module~$P_i$ is finitely generated for all~$i\le n$.
For example, $M$ is of type~$\sfFP_0$ (resp.~$\sfFP_1$) if and only if~$M$ is finitely generated (resp.\ finitely presented) as an $R$-module.
	
A group~$G$ is said to be \emph{of type~$\sfFP_n(R)$} if the trivial $RG$-module~$R$ is of type~$\sfFP_n$.
For example, every group is of type~$\sfFP_0(R)$, and a group~$G$ is of type~$\sfFP_1(\IZ)$ if and only if~$G$ is finitely generated.
Clearly, a group of type~$\sfF_n$ is of type~$\sfFP_n(\IZ)$.
However, there exist groups of type~$\sfFP_2(\IZ)$ that are not of type~$\sfF_2$ (i.e., not finitely presented)~\cite{Bestvina-Brady97}.
For~$n\ge 2$, a group~$G$ is of type~$\sfF_n$ if and only if~$G$ is finitely presented and of type~$\sfFP_n(\IZ)$.
	
The algebraic finiteness properties of a subgroup~$H$ of~$G$ are equivalent to the finiteness properties of the permutation $RG$-module~$R[G/H]$.
	
\begin{lem}
\label{lem:permutation}
	Let~$G$ be a group, let~$H$ be a subgroup of~$G$, and let~$n\in \IN$.
	Then the group~$H$ is of type~$\sfFP_n(R)$ if and only if the $RG$-module~$R[G/H]$ is of type~$\sfFP_n$.
	\begin{proof}
		The induction functor~$\Ind_H^G$ preserves exactness, projectivity, and finite generation.
		Hence, if the $RH$-module~$R$ is of type~$\sfFP_n$, then the $RG$-module~$R[G/H]\cong \Ind_H^G R$ is of type~$\sfFP_n$.
		Conversely, suppose that the $RG$-module~$R[G/H]$ is of type~$\sfFP_n$.
		By induction on~$n$, we may assume  that the group~$H$ is of type~$\sfFP_{n-1}(R)$.
		 Then there exists an exact sequence of~$RH$-modules
		\[
			0\to K\to P_{n-1}\to P_{n-2}\to \cdots\to P_0\to R\to 0,
		\]
		where the $RH$-module~$P_i$ is finitely generated projective for all~$i\le n-1$.
		By applying the induction functor~$\Ind_H^G$, we obtain an exact sequence of $RG$-modules
		\[
			0\to \Ind_H^GK\to \Ind_H^G P_{n-1}\to \Ind_H^G P_{n-2}\to \cdots\to \Ind_H^G P_0\to R[G/H]\to 0, 
		\]
		where the $RG$-module~$\Ind_H^G P_i$ is finitely generated projective for all~$i\le n$.
		Since the $RG$-module~$R[G/H]$ is of type~$\sfFP_n$, the $RG$-module~$\Ind_H^G K$ is finitely generated.
		Since the ring extension $RH\to RG$ is faithfully flat, it follows that the $RH$-module~$K$ is finitely generated.
		 Hence the group~$H$ is of type~$\sfFP_n(R)$.
	\end{proof}
\end{lem}

Algebraic finiteness properties of modules satisfy 2-out-of-3 properties.
	
\begin{lem}[{\cite[Proposition~1.4]{Bieri81}}]
\label{lem:extensions}
	Let~$0\to K\to M\to N\to 0$ be a short exact sequence of $R$-modules and let~$n\in \IN$.
	The following hold:
	\begin{enumerate}[label=\enum]
		\item\label{item:quotient} If~$K$ is of type~$\sfFP_{n-1}$ and~$M$ is of type~$\sfFP_n$, then~$N$ is of type~$\sfFP_n$;
		\item\label{item:extension} If~$K$ and~$N$ are of type~$\sfFP_n$, then~$M$ is of type~$\sfFP_n$;
		\item\label{item:kernel} If~$M$ is of type~$\sfFP_{n-1}$ and~$N$ is of type~$\sfFP_n$, then~$K$ is of type~$\sfFP_{n-1}$.
	\end{enumerate}
\end{lem}

We use Lemma~\ref{lem:extensions} by cutting exact sequences into short exact sequences.

\begin{lem}
\label{lem:chain}
	Let~$C_*$ be an $R$-chain complex that is concentrated in degrees~$\ge -1$ and let~$n\in \IN$.
	Suppose that the following hold:
	\begin{enumerate}[label=\enum]
		\item For all~$i\in \{-1,\ldots,n\}$, $H_i(C_*)=0$;
		\item For all~$i\in \{1,\ldots,n+1\}$, $C_i$ is of type~$\sfFP_{n+1-i}$;
		\item $C_{-1}$ is of type~$\sfFP_n$.
	\end{enumerate}
		Then~$C_0$ is of type~$\sfFP_n$.
	\begin{proof}
		For all~$i\in \IN$, we denote~$K_i\coloneqq \ker(\partial_i\colon C_i\to C_{i-1})$.
		The following horizontal sequence of $R$-modules is exact by assumption~(i):
		\[\begin{tikzcd}
			C_{n+1}\ar{r}{\partial_{n+1}}\ar[two heads]{dr} 
			& C_n\ar{r}{\partial_n}\ar[two heads]{dr} 
			& C_{n-1}\ar{r}{\partial_{n-1}} 
			& \cdots\ar{r}{\partial_1}\ar[two heads]{dr} 
			& C_0\ar{r}{\partial_0} 
			& C_{-1}\ar{r} 
			& 0. 
			\\
			& K_n\ar[hook]{u} 
			& K_{n-1}\ar[hook]{u} 
			& \cdots 
			& K_0\ar[hook]{u}
		\end{tikzcd}\]
		Since~$C_{n+1}$ is of type~$\sfFP_0$ (i.e., finitely generated) by assumption~(ii), also~$K_n$ is of type~$\sfFP_0$.
		Since~$C_n$ is of type~$\sfFP_1$ by assumption~(ii) and~$K_n$ is of type~$\sfFP_0$, it follows that~$K_{n-1}$ is of type~$\sfFP_1$ by Lemma~\ref{lem:extensions}~\ref{item:quotient}.
		Inductively, we obtain that~$K_0$ is of type~$\sfFP_n$.
		Since~$C_{-1}$ is of type~$\sfFP_{n}$ by assumption~(iii) and~$K_0$ is of type~$\sfFP_n$, it follows that~$C_0$ is of type~$\sfFP_n$ by Lemma~\ref{lem:extensions}~\ref{item:extension}.
	\end{proof}
\end{lem}

Applying Lemma~\ref{lem:chain} to the cellular chain complex of a $G$-CW-complex yields a criterion for the algebraic finiteness properties of vertex stabilisers.
We denote the stabiliser of a cell~$\sigma$ by $G_\sigma\coloneqq \{g\in G\mid g\sigma=\sigma\}$. 

\begin{prop}
\label{prop:FPn}
	Let~$G$ be a group and let~$X$ be a $G$-CW-complex.
	Let~$R$ be a ring and let~$n\in \IN$. 
	Suppose that the following hold:
	\begin{enumerate}[label=\enum]
		\item $X$ is $n$-acyclic over~$R$ (i.e., $H_i(X;R)\cong H_i(\operatorname{pt};R)$ for all~$i\le n$);
		\item The quotient~$G\backslash X^{(n+1)}$ of the $(n+1)$-skeleton is compact;
		\item For all~$i\in \{1,\ldots,n\}$, all stabilisers of $i$-cells of~$X$ are of type~$\sfFP_{n+1-i}(R)$;
		\item $G$ is of type~$\sfFP_n(R)$.
	\end{enumerate}
	Then all stabilisers of 0-cells of~$X$ are of type~$\sfFP_n(R)$.
\begin{proof}
	Let~$C_*$ be the augmented cellular $RG$-chain complex of~$X$
	(with~$C_{-1}=R$ and the usual augmentation map $\partial_0\colon C_0\to C_{-1}$).
	For all~$i\in \IN$, the $RG$-module~$C_i$ is of the form~$\bigoplus_{\sigma\in \Sigma_i}R[G/G_\sigma]$, where~$\Sigma_i$ is a set of $G$-orbit representatives of the $i$-cells of~$X$.
	For all~$i\le n+1$, the set~$\Sigma_i$ is finite by assumption~(ii).
	We have:
	\begin{itemize}
		\item For all~$i\in \{-1,\ldots,n\}$, $H_i(C)=0$ by assumption~(i);
		\item For all~$i\in \{1,\ldots,n+1\}$, $C_i$ is of type~$\sfFP_{n+1-i}$ by assumptions~(ii) and~(iii);
		\item $C_{-1}$ is of type~$\sfFP_n$ by assumption~(iv).
	\end{itemize}
	Then Lemma~\ref{lem:chain} applies to the $RG$-chain complex~$C_*$ and yields that the $RG$-module~$C_0\cong \bigoplus_{v\in \Sigma_0}R[G/G_v]$ is of type~$\sfFP_n$.
	
	By induction on~$n$, we may assume that for every $0$-cell~$v$ of~$X$, the stabiliser~$G_v$ is of type~$\sfFP_{n-1}(R)$. 
	In particular, the $RG$-modules~$(R[G/G_v])_v$ are of type~$\sfFP_{n-1}$.
	Since~$C_0$ is of type~$\sfFP_n$, it follows that the $RG$-modules~$(R[G/G_v])_v$ are of type~$\sfFP_n$ by Lemma~\ref{lem:extensions}~\ref{item:quotient}.
	Hence the groups~$(G_v)_v$ are of type~$\sfFP_n(R)$ by Lemma~\ref{lem:permutation}.
	\end{proof}
\end{prop}

When~$X$ is 1-dimensional,
Proposition~\ref{prop:FPn} is classical~\cite[Proposition~2.13]{Bieri81}.
As an application to higher dimensions, we recover a result of Dahmani--Guirardel--Osin~\cite{DGO17} for relatively hyperbolic groups.

\begin{ex}
	Let~$G$ be a finitely generated group and let~$\calH$ be a finite collection of subgroups of~$G$.
	Suppose that the pair~$(G,\calH)$ is relatively hyperbolic.
	If~$G$ is of type~$\sfFP_n(R)$, then every~$H\in \calH$ is of type~$\sfFP_n(R)$ ~\cite[Corollary~4.32]{DGO17}.
	This result also follows by applying Proposition~\ref{prop:FPn} to the Rips complex (for sufficiently large diameter) on the coned-off Cayley graph~\cite{MPP19}.
	It is a contractible cocompact $G$-CW-complex such that all stabilisers of positive dimensional cells are finite and every~$H\in \calH$ appears as a vertex stabiliser. 
\end{ex}

\begin{rem}
	The proof of Proposition~\ref{prop:FPn} relies mainly on 2-out-of-3 properties (Lemma~\ref{lem:extensions}) and lends itself to generalisation.
	The algebraic analogue of Theorem~\ref{thm:combination} for~$\sfFP_n$~\cite[Proposition~1.1]{Brown87} can be proved in a similar manner using only Lemma~\ref{lem:extensions}~\ref{item:quotient}.
	Note that to prove Proposition~\ref{prop:FPn} we used only Lemma~\ref{lem:extensions}~\ref{item:quotient} and~\ref{item:extension}.
	Using Lemma~\ref{lem:extensions}~\ref{item:quotient},~\ref{item:extension}, and~\ref{item:kernel} one can prove a similar criterion for stabilisers of $d$-cells, for~$d\ge 2$, under appropriate assumptions on stabilisers of cells of all other dimensions.
	
	Moreover, apart from algebraic finiteness properties, our method of proof applies to other homological group properties, such as vanishing of cohomology with coefficients in the group ring, being a duality group, and vanishing of $\ell^2$-Betti numbers.
\end{rem}

\section{Finite presentability of vertex stabilisers}

We do not know if the topological analogue of Proposition~\ref{prop:FPn} for $n=2$ holds.

\begin{question}
\label{qn:F2}
	Let~$G$ be a group and let~$X$ be a $G$-CW-complex.
	Suppose that the following hold:
	\begin{enumerate}[label=\enum]
		\item $X$ is $2$-connected;
		\item The quotient~$G\backslash X^{(3)}$ of the $3$-skeleton is compact;
		\item All stabilisers of 2-cells of~$X$ are finitely generated;
		\item All stabilisers of 1-cells of~$X$ are finitely presented;
		\item $G$ is finitely presented.
	\end{enumerate}
	Then are all stabilisers of 0-cells of~$X$ finitely presented?
\end{question}

When~$X$ is 1-dimensional, the answer is positive~\cite{Haglund-Wise21} (see also~\cite[Theorem~3.1]{Dydak82}).
We will answer Question~\ref{qn:F2} affirmatively under additional assumptions on links of vertices (Proposition~\ref{prop:F2}), generalising the 1-dimensional case.

We use a well-known ``blow-up" construction that produces from a $G$-CW-complex a \emph{free} $G$-CW-complex and a (non-equivariant) homotopy equivalence with prescribed fibres.

\begin{thm}[\cite{Haefliger92}]
\label{thm:blow-up}
	Let~$G$ be a group and let~$X$ be a $G$-CW-complex.
	Let~$\Sigma$ be a set of $G$-orbit representatives of the open cells of~$X$.
	For every cell~$\sigma\in \Sigma$, we fix a model for~$EG_\sigma$.
	Then there exists a free $G$-CW-complex~$\widehat{X}$ and a $G$-map~$p\colon \widehat{X}\to X$ such that the following hold:
	\begin{enumerate}[label=\enum]
		\item For every open cell~$\sigma\in \Sigma$, there is a $G$-homeomorphism $p^{-1}(\sigma)\cong EG_\sigma\times \sigma$;
		\item For every $G$-subcomplex~$A$ of~$X$, the restriction $p|_{p^{-1}(A)}\colon p^{-1}(A)\to A$ is a (non-equivariant) homotopy equivalence.
		In particular, $p\colon \widehat{X}\to X$ is a (non-equivariant) homotopy equivalence.
	\end{enumerate}
\end{thm}

Note that Theorem~\ref{thm:combination} is a direct consequence of Theorem~\ref{thm:blow-up}.

For a finitely presented group~$G$, we extract from a simply-connected free $G$-CW-complex a simply-connected \emph{cocompact} $G$-subcomplex.

\begin{lem}
\label{lem:deleting cells}
	Let~$G$ be a finitely presented group and let~$X$ be a simply-connected free $G$-CW-complex with cocompact 1-skeleton.
	Let~$F$ be a finite set of $G$-orbits of 2-cells of~$X$.
	Then there exists a simply-connected 2-dimensional cocompact $G$-subcomplex~$Z$ of~$X$ containing~$F$ and the 1-skeleton of~$X$.
	\begin{proof}
		The quotient CW-complex~$G\backslash X$ has compact 1-skeleton and fundamental group~$G$.
		The CW-structure on~$G\backslash X$ yields a group presentation $G\cong \spann{S\mid R}$ with finite generating set~$S$.
		The finite set~$F$ corresponds to a finite subset~$R_F$ of~$R$.
		Since~$G$ is finitely presented, there exists a finite subset~$R'$ of~$R$ containing~$R_F$ such that $G\cong \spann{S\mid R'}$.
		Let~$Y$ be the 2-dimensional compact subcomplex of~$G\backslash X$ consisting of the 1-skeleton and those 2-cells corresponding to~$R'$.
		Denoting by $q\colon X\to G\backslash X$ the quotient map,
		the preimage $Z\coloneqq q^{-1}(Y)$ is as desired.
	\end{proof}
\end{lem}

For a simplicial $G$-CW-complex~$X$,
the stabiliser~$G_v$ of a vertex~$v$ acts on the link of~$v$ in~$X$.
If the link is simply-connected, then the finite presentability of~$G_v$ can be deduced by applying Theorem~\ref{thm:combination} for~$n=2$ to the link.
Below we assume only that every connected component of the link is simply-connected.
This is the case, e.g., when~$X$ is 1-dimensional.

\begin{prop}
\label{prop:F2}
	Let~$G$ be a group and let~$X$ be a simplicial $G$-CW-complex.
	Suppose that the following hold:
	\begin{enumerate}[label=\enum]
		\item $X$ is simply-connected;
		\item The quotient~$G\backslash X^{(3)}$ of the 3-skeleton is compact;
		\item All stabilisers of 2-simplices of~$X$ are finitely generated;
		\item All stabilisers of 1-simplices of~$X$ are finitely presented;
		\item For every 0-simplex~$v$ of~$X$, every connected component of the link of~$v$ in~$X$ is simply-connected;
		\item $G$ is finitely presented.
	\end{enumerate}
	Then all stabilisers of 0-simplices of~$X$ are finitely presented.
	\begin{proof}
		By Proposition~\ref{prop:FPn} for~$n=1$, all vertex stabiliers are of type~$\sfFP_1(\IZ)$ (i.e., finitely generated).
		Let~$v$ be a vertex of~$X$.
		We will show that the stabiliser~$G_v$ is finitely presented by exhibiting a simply-connected 2-dimensional $G_v$-cocompact free $G_v$-CW-complex.
		
		By replacing~$X$ with its 3-skeleton~$X^{(3)}$, we may assume that~$X$ is 3-dimensional.
		We subdivide~$X$, vaguely speaking, by inserting a small sphere centred at~$v$.
		More precisely, first, we subdivide each edge~$e$ containing~$v$ by inserting a new vertex~$w_e$ close to~$v$. 
		Second, we subdivide each 2-simplex~$f$ containing~$v$ as follows:
		For two edges~$e_1,e_2$ of~$f$ meeting at~$v$, we insert a new edge~$e_f$ connecting the new vertices~$w_{e_1},w_{e_2}$.
		Third, we subdivide each 3-simplex~$\sigma$ containing~$v$ as follows: 
		For three faces~$f_1,f_2,f_3$ of~$\sigma$ meeting at~$v$, we insert a new 2-simplex~$f_\sigma$ spanned by the new edges~$e_{f_1},e_{f_2},e_{f_3}$. 
		By performing this subdivision $G$-equivariantly, i.e., at all $G$-translates of~$v$, we obtain a (not necessarily simplicial) $G$-CW-complex~$X'$, where the new vertex~$w_e$ has stabiliser~$G_e$, the new edge~$e_f$ has stabiliser~$G_f$, and the new 2-cell~$f_\sigma$ has stabiliser~$G_\sigma$.
		
		Let~$\lk_{X'}(v)$ denote the (non-equivariant) subcomplex of~$X'$ consisting of the new cells~$w_e,e_f,f_\sigma$ corresponding to simplices of~$X$ containing~$v$.
		Let~$\st_{X'}(v)$ denote the (non-equivariant) subcomplex of~$X'$ consisting of~$\lk_{X'}(v)$ and all cells of~$X'$ containing~$v$.
		Then~$\lk_{X'}(v)$ and~$\st_{X'}(v)$ are (non-equivariantly) homotopy equivalent to the link and the star of~$v$ in~$X$, respectively.
		Note that~$G_v$ acts cocompactly on~$\lk_{X'}(v)$ and on~$\st_{X'}(v)$ by assumption~(ii).
		Since we performed the subdivision in a small neighbourhood of~$v$ and its $G$-translates, $X'$ contains the induced $G$-CW-complexes~$G\times_{G_v} \lk_{X'}(v)$ and~$G\times_{G_v}\st_{X'}(v)$ as $G$-subcomplexes.
				
		For every ($G$-orbit representative) vertex~$w$, edge~$e$, 2-simplex~$f$, and 3-simplex~$\sigma$ of the simplicial $G$-CW-complex~$X$, we fix a model for~$EG_w$, $EG_e$, $EG_f$, and~$EG_\sigma$ with cocompact 1-skeleton, 2-skeleton, 1-skeleton, and 0-skeleton, respectively, using assumptions~(iii) and~(iv).
		Theorem~\ref{thm:blow-up} yields a  blow-up $p\colon \widehat{X'}\to X'$ with respect to the chosen classifying spaces.
		Since~$p$ is a (non-equivariant) homotopy equivalence
		and~$X'$ is simply-connected by assumption~(i), also~$\widehat{X'}$ is simply-connected.
		By construction, the $G$-CW-complex~$\widehat{X'}$ is free and has $G$-cocompact 1-skeleton, and the $G$-subcomplex~$p^{-1}(G\times_{G_v} \lk_{X'}(v))$ has $G$-cocompact 2-skeleton.
		Since~$G$ is finitely presented by assumption~(vi), Lemma~\ref{lem:deleting cells} yields a simply-connected 2-dimensional $G$-cocompact $G$-subcomplex~$Z$ of~$\widehat{X'}$ containing the 2-skeleton of~$p^{-1}(G\times_{G_v} \lk_{X'}(v))$ and the 1-skeleton of~$\widehat{X'}$.		
		We denote $L_Z\coloneqq Z\cap p^{-1}(\lk_{X'}(v))= p^{-1}(\lk_{X'}(v))^{(2)}$ and $S_Z\coloneqq Z\cap p^{-1}(\st_{X'}(v))$.
		
		The 2-dimensional free $G_v$-CW-complex~$S_Z$ is connected and $G_v$-cocompact.
		In order to conclude that~$G_v$ is finitely presented, it remains to show that~$S_Z$ is simply-connected.
		Since~$Z$ is simply-connected,
		the complement of~$S_Z$ in~$Z$ has one connected component~$K_i$ for each connected component~$L_i$ of~$L_Z$. 
		By properties of the map~$p$ (Theorem~\ref{thm:blow-up}), the preimage~$p^{-1}(G\times_{G_v}\lk_{X'}(v))$ is (non-equivariantly) homotopy equivalent to~$G\times_{G_v}\lk_{X'}(v)$, whose connected components are simply-connected by assumption~(v).
		It follows that~$L_i$, as the 2-skeleton of a connected component of~$p^{-1}(G\times_{G_v} \lk_{X'}(v))$, is simply-connected.
		Fix a basepoint of~$Z$ lying in~$S_Z$.
		Consider the cover of~$Z$ by open thickenings of $S_Z$ and the~$(K_i)_i$ with whiskers to the basepoint. 
		Then the Seifert--van Kampen Theorem for open covers yields that~$\pi_1(S_Z)$ is a free factor of~$\pi_1(Z)$.
		Since~$\pi_1(Z)$ is trivial by construction of~$Z$, we conclude that~$\pi_1(S_Z)$ is trivial.
	\end{proof}
\end{prop}

\begin{proof}[Proof of Theorem~\ref{thm:main Fn}]
	For~$n=1$, this follows from Proposition~\ref{prop:FPn}, since groups of type~$\sfFP_1(\IZ)$ are of type~$\sfF_1$.
	For~$n\ge 2$, this follows by combining Proposition~\ref{prop:FPn} and Proposition~\ref{prop:F2}, since finitely presented groups of type~$\sfFP_n(\IZ)$ are of type~$\sfF_n$.
\end{proof}

\section{Uncountably many groups of type~$\sfFP_n$ and not of type~$\sfFP_{n+1}$}

We conclude this note with an application of Proposition~\ref{prop:FPn}.
For~$n\in \IN$, we define the class of groups
\[
	\calC_n\coloneqq \{ \text{groups of type~$\sfFP_n(\IZ)$ and not of type~$\sfFP_{n+1}(\IZ)$}\}.
\]
An example of a group in~$\calC_n$ is given by the kernel of the group homomorphism $\prod^{n+1}(\IZ\ast \IZ)\to \IZ$ that sends every generator (in the obvious generating set) to $1\in \IZ$.
This group is finitely presented for~$n\ge 2$.

\begin{lem}
\label{lem:one-ended}
	Let~$n\ge 2$ and let~$G$ be a group in~$\calC_n$.
        Then~$G$ has a one-ended subgroup~$H$ in~$\calC_n$.
        Moreover, if~$G$ is finitely presented, then~$H$ can be taken to be finitely presented. 
	\begin{proof}
		Since groups of type~$\sfFP_2(\IZ)$ are accessible~\cite{Dunwoody85}, the group~$G$ is the fundamental group of a finite graph of groups whose edge groups are finite and vertex groups are finite or one-ended.
		Since~$G$ is of type~$\sfFP_n(\IZ)$, every vertex group is of type~$\sfFP_n(\IZ)$ by Proposition~\ref{prop:FPn}.
            Moreover, if~$G$ is finitely presented, then every vertex group is finitely presented by Corollary~\ref{cor:Haglund-Wise}.
		On the other hand, since~$G$ is not of type~$\sfFP_{n+1}(\IZ)$, there is a one-ended vertex group that is not of type~$\sfFP_{n+1}(\IZ)$ by the algebraic version of Theorem~\ref{thm:combination}~\cite[Proposition~1.1]{Brown87}.
	\end{proof}
\end{lem}

Finiteness properties are quasi-isometry invariants~\cite{Alonso94}.
It is natural to ask how many (one-ended) groups there are in~$\calC_n$, up to quasi-isometry.
For \mbox{$n\ge 1$}, countably infinitely many one-ended groups in~$\calC_n$ that are pairwise not quasi-isometric were exhibited by Patil~\cite[Theorem~4.2]{Patil23}.
For~$n\ge 2$, we improve this result to uncountably many groups.
We construct the desired groups as small cancellation quotients of free products relying on work of Mart\'inez-Pedroza and the second author~\cite{MPSS24}.

\begin{thm}
\label{thm:uncountably many}
	Let $n\geq 2$. 
	The class~$\calC_n$ contains uncountably many one-ended groups that are pairwise not quasi-isometric.
	\begin{proof}
		Let~$A$ be a finitely presented one-ended group in~$\calC_n$ (Lemma~\ref{lem:one-ended}). 
		Let~$(B_i)_{i\in \IR}$ be the collection of one-ended groups of type~$\sfFP(\IZ)$ that are pairwise not quasi-isometric provided by Kropholler--Leary--Soroko~\cite[Corollary~1.3]{KLS20}.
		For every~$i\in \IR$, let~$\calR_i$ be a finite symmetrised subset of~$A\ast B_i$ satisfying the $C'(1/12)$ small cancellation condition and consider the group
		\[
			G_i\coloneqq (A\ast B_i)/\langle\langle \calR_i \rangle\rangle.
		\]
		Let~$i\in \IR$. 
		Since~$A$ and~$B_i$ are one-ended, the group~$G_i$ is one-ended~\cite[Theorem~5.1]{MPSS24}.
		There exists a contractible 2-dimensional cocompact $G_i$-CW-complex whose stabilisers of 1- and 2-cells are finite and vertex stabilisers are conjugates of~$A$ and of~$B_i$~\cite[Proposition~4.3]{MPSS24}.
		Then, since~$A$ is of type~$\sfFP_n(\IZ)$ and~$B_i$ is of type~$\sfFP(\IZ)$, the group~$G_i$ is of type~$\sfFP_n(\IZ)$~\cite[Proposition~1.1]{Brown87}.
		On the other hand, since~$A$ is not of type~$\sfFP_{n+1}(\IZ)$, the group~$G_i$ is not of type~$\sfFP_{n+1}(\IZ)$ by Proposition~\ref{prop:FPn}.
		Together, the group~$G_i$ lies in~$\calC_n$.

		It remains to show that the groups~$G_i$ and~$G_j$ are not quasi-isometric for~$i,j\in \IR$ with~$i\neq j$.
		We use a quasi-isometry invariant due to Bowditch~\cite{Bowditch98}, called the \emph{taut loop length spectrum}, that assigns to a finitely generated group~$G$ a subset~$H(G)\subset \IN$, up to a suitable relation~$\sim$ on the power set of~$\IN$.
		It follows from the definitions that~$H(G)\sim \emptyset$ if~$G$ is finitely presented and that the relation~$\sim$ is compatible with taking finite unions.
		By~\cite[Theorem~2.4]{MPSS24}, we have $H(G_i)\sim H(A)\cup H(B_i)$ for all~$i\in \IR$.
		Since~$A$ is finitely presented, we  obtain $H(G_i)\sim H(B_i)$.
		By construction of the groups~$(B_i)_{i\in \IR}$, we have $H(B_i)\not\sim H(B_j)$ for~$i\neq j$~\cite[Section~6]{KLS20}.
            Thus, we conclude $H(G_i)\not\sim H(G_j)$ for~$i\neq j$.
	\end{proof}
\end{thm}

\bibliographystyle{alpha}
\bibliography{bib}

\setlength{\parindent}{0cm}

\end{document}